\documentclass[12pt,lettersize]{amsart}

  \usepackage{amsthm} 
  \usepackage{amsmath} 
  \usepackage{amssymb} 
  \usepackage{mathtools} 
  \usepackage{amsfonts}
  \usepackage{amscd,url,caption,dsfont,braket,csquotes,enumitem,microtype,bm}
  \usepackage[all]{xy}
  \usepackage[pdftex]{graphicx}

  \usepackage[usenames,dvipsnames]{xcolor}
  \usepackage{tikz}
  \usepackage{pgfplots}

  \usepackage[cal=boondox]{mathalfa}
  
  \usepackage{xspace}
  \usepackage{pinlabel}
  \usepackage{enumerate}
  \usepackage{changepage}
  \usepackage{scrextend}
  \usepackage{appendix}
  \usepackage{float}

\usepackage[margin=1.3in]{geometry}
  \usepackage{hyperref}

  \newcommand{\calA}{\mathcal{A}}

  \newcommand{\calG}{\mathcal{G}}

  \newcommand{\calM}{\mathcal{M}}
  \newcommand{\calN}{\mathcal{N}}

  \newcommand{\calS}{\mathcal{S}}
  
  \newcommand{\calU}{\mathcal{U}}

  \newcommand{\NN}{\mathbb{N}}

  \newcommand{\RR}{\mathbb{R}}

  \newcommand{\ZZ}{\mathbb{Z}}

  \newtheorem{theorem}{Theorem}[section]
  \newtheorem{proposition}[theorem]{Proposition}
  \newtheorem{corollary}[theorem]{Corollary}
  \newtheorem{lemma}[theorem]{Lemma}

  \newtheorem{question}[theorem]{Question}

  \theoremstyle{definition}
  \newtheorem{definition}[theorem]{Definition}

  \newtheorem*{claim*}{Claim}
  
  \newtheorem{example}[theorem]{Example}
  \newtheorem{fact}[theorem]{Fact}
  \newtheorem*{question*}{Question}
  \newtheorem*{answer*}{Answer}
  \newtheorem*{application*}{Application}

  \theoremstyle{remark}
  \newtheorem{remark}[theorem]{Remark}
  \newtheorem*{remark*}{Remark}
  
  \DeclareMathOperator{\id}{id}

  \newcommand{\Map}{\ensuremath{\operatorname{Map}}\xspace}    
   
  \newcommand{\param}{{\mathchoice{\mkern1mu\mbox{\raise2.2pt\hbox{$
  \centerdot$}}
  \mkern1mu}{\mkern1mu\mbox{\raise2.2pt\hbox{$\centerdot$}}\mkern1mu}{
  \mkern1.5mu\centerdot\mkern1.5mu}{\mkern1.5mu\centerdot\mkern1.5mu}}}

  \renewcommand{\setminus}{{\smallsetminus}}
  \newcommand{\ssm}{\smallsetminus}

  \newcommand{\from}{\colon\thinspace} 
  
  \newcommand{\eps}{\varepsilon}
  \newcommand{\vphi}{\varphi}
  \newcommand{\del}{\partial}

  \DeclarePairedDelimiter\abs{\lvert}{\rvert}

  \newcommand{\mute}[1]{}

\begin{document}

\title{The Grand Arc Graph}
\author{Assaf Bar-Natan and Yvon Verberne}
\date{\today}

\maketitle
\begin{abstract}
In this article, we construct a new simplicial complex for infinite-type surfaces, which we call the \textit{grand arc graph}.
We show that if the end space of a surface has at least three different self-similar equivalence classes of maximal ends, then the grand arc graph is infinite-diameter and $\delta$-hyperbolic. In this case, we also show that the mapping class group acts on the grand arc graph by isometries and acts on the visible boundary continuously. When the surface has stable maximal ends, we also show that this action has finitely many orbits.
\end{abstract}

\vspace{-0.25cm}

\section{Introduction}

Let $\Sigma$ be a connected, orientable, infinite-type surface.
The mapping class group of $\Sigma$, denoted by $\Map(\Sigma)$, is the group of orientation preserving homeomorphisms of $\Sigma$, up to isotopy. 
Informally, an \textit{end} of an infinite-type surface is a way of exiting every compact subsurface in the sense of Freudenthal \cite{Fr31}.

In \cite{MR}, Mann and Rafi define a partial order on the space of ends, allowing one to consider the space of \textit{maximal ends}.
In this paper, we will consider surfaces which have finitely many self-similar equivalence classes of maximal ends. We call this splitting of 
the end-space a \textit{grand splitting}, denoted $\calS(\Sigma)$, so that 
the space of maximal ends is 
equal to $\cup_{E\in \calS(\Sigma)} E$. Since the grand splitting is preserved by $\Map(\Sigma)$, we are motivated to use it to construct a graph upon which the mapping class group acts by isometries. 

A grand arc is a simple arc in the surface $\Sigma$ such that the arc converges to exactly two ends, each of which are in different elements of the grand splitting.
We let the vertices of the grand arc graph, $\calG(\Sigma)$, be isotopy classes of grand arcs, and we place an edge between two vertices if there exists disjoint representatives of the grand arcs. We show that for many surfaces, the grand arc graph is $\delta$-hyperbolic in the sense of Gromov and has infinite-diameter.\\

\begin{theorem} ~\label{thm:InfDiamHyperbolic}
Let $\Sigma$ be an infinite-type surface with $|\calS(\Sigma)|<\infty$, 
which is not the once-punctured Cantor 
tree. Then $\calG(\Sigma)$ is a nonempty, 
connected, infinite-diameter $\delta$-hyperbolic metric space if and only if $|\calS(\Sigma)| \ge 3$. Moreover, the constant $\delta$ is independent of $\Sigma$.
\end{theorem}

Theorem ~\ref{thm:InfDiamHyperbolic} does hold in the case of the once punctured Cantor tree by the work of Bavard \cite{Bav}, which uses different tools than those found in this paper. In Section \ref{Sec:Hyperbolicity}, we show that when $|\calS(\Sigma)| = 2$, the grand arc graph is not hyperbolic, but is still of infinite diameter. When $|\calS(\Sigma)| = 1$, the grand arc graph is empty.

For a finite-type surface $S$, the \textit{curve graph} is a combinatorial tool used to study the mapping class group.
The curve graph was introduced by Harvey \cite{Harvey} and is defined to be the simplicial complex where every vertex is an isotopy class of an essential simple closed curve, and an edge between two vertices signifies disjointness. Masur-Minsky proved that the curve graph is hyperbolic \cite{MM1}, and also used the curve graph to study the large-scale geometry of the mapping class group \cite{MM2}.
The utility of the curve graph to aid in the study of mapping class groups of finite-type surfaces. Unfortunately, the curve graph of an infinite-type surface is of diameter $2$, which leads one to ask the following question.

\begin{question}[AIM Workshop Problem 2.1]\label{Question:AimCurveComplexAnalogue}
What combinatorial objects are ``good" analogues of the curve graph, either uniformly for all infinite-type surfaces or for some class of infinite-type surfaces? Here
``good" means that there exist relationships between topological properties of the mapping class and
dynamical properties of its action on the combinatorial object.
\end{question}

The grand arc graph can be defined for every infinite-type surface with maximal ends.
However, we will find that the grand arc graph has the desired properties outlined in Question \ref{Question:AimCurveComplexAnalogue} in the case where there are at least 3 but not infinitely-many self-similar equivalence classes of maximal ends.
For these cases, the grand arc graph is a reasonable candidate for answering Question 
\ref{Question:AimCurveComplexAnalogue}.

In the case of surfaces of finite-type, the mapping class group has the discrete topology and acts on the curve graph continuously by isometries.
By using the work of Fanoni--Ghaswala--McLeay \cite{FGM}, it can be shown that 
the mapping class group of an infinite-type surface does not act by continuous isometries on the grand arc graph. 
However, we can prove that when the grand arc graph is $\delta$-hyperbolic, the action of the mapping class group on the boundary of the grand arc graph is continuous. We prove this result by showing that the action of the mapping class group on $\calG(\Sigma)$ is \textit{$3$-quasi-continuous}, a coarse relaxation of continuity.

\begin{theorem} \label{thm:mcgaction}
Let $\Sigma$ be an infinite-type surface with 
$|\calS(\Sigma)| < \infty$. Then the mapping class group acts on $\calG(\Sigma)$ by isometries. Moreover, if $|\calS(\Sigma)| > 2$, then the action on $\del\calG(\Sigma)$ is continuous. Moreover, when the maximal ends of $\Sigma$ are stable, then the mapping class group acts on $\calG(\Sigma)$ with finitely many orbits.
\end{theorem}

\subsection{Prior results}

Throughout  this section, let $\Sigma$ be an 
infinite-type surface, and let $E$ be its end space. Let $\calG(\Sigma)$ be the grand arc graph associated to $\Sigma$.

\subsubsection{Relationship to the ray graph}

The \textit{ray graph} was introduced by Calegari in \cite{Calegari} and is defined as follows. The ray graph is the graph whose vertex set is the
set of isotopy classes of proper rays, with interior in the complement of the Cantor
set $K\subset \RR^2$, from a point in $K$ to infinity, and whose edges (of length 1) are the pairs of
such rays that can be realized disjointly.

In 2016, Bavard showed that the ray graph is $\delta$-hyperbolic, and used the ray graph to show the existence of non-trivial quasimorphisms on the once-punctured Cantor tree.
We note that the definition of the ray graph and the definition of the grand arc graph of the once-punctured Cantor tree agree, so in some sense, the grand arc graph is a generalization of the ray graph but for surfaces which have more than two types of maximal ends. For the ray graph, the point at infinity can be thought of as an isolated puncture, and therefore, $\RR^2 \smallsetminus K$ is the once-punctured Cantor tree surface. All ends of the once-punctured Cantor tree are maximal, and the ends comprising of the Cantor set are all in the same self-similar equivalence class of maximal ends. Therefore, the ray graph is precisely the grand arc graph for the once-punctured Cantor tree. In this sense, the grand arc graph is a generalization of the ray graph, but for surfaces which have more than two types of maximal ends.

\subsubsection{Relationship to the omnipresent arc graph}
 The omnipresent arc graph was introduced by Fanoni, Ghaswala and McLeay \cite{FGM} and is defined as follows.
A \textit{one-cut surface} is a complementary component of a separating loop. 
A one-cut subsurface which is homeomorphic to the entire surface is called a \textit{one-cut homeomorphic subsurface}.
An \textit{omnipresent arc} is an arc which joins two distinct ends and intersects all one-cut homeomorphic subsurfaces.
The \textit{omnipresent arc graph} is the graph where vertices are homotopy classes of omnipresent arcs, and edges correspond to having disjoint representatives.

Fanoni--Ghaswala--McLeay show that if a surface $\Sigma$ is stable (see Definitions 5.3 and 5.4, and Theorem B in \cite{FGM}), then an arc is omnipresent if and only if it joins ends whose orbit under $\Map(\Sigma)$ is finite. Additionally, note that finite orbit ends must be maximal, and by self-similarity of the sets in the grand splitting, they must lie in different elements of the grand splitting. In other words: 

\begin{fact}
If $\Sigma$ is a stable surface, and $\alpha$ is an omnipresent arc in $\Sigma$, then $\alpha$ is a grand arc.
\end{fact}

In the general case when $\Sigma$ 
is not stable, it is not necessarily true that omnipresent arcs are grand\footnote{This fact was communicated to the first author by Kasra Rafi, and follows from a construction of a surface with a single, unstable maximal end.}.
We also note that grand arcs aren't necessarily omnipresent.
For instance, consider the surface which is
the connect sum of a Cantor tree and a 
blooming Cantor tree where every end is
accumulated by genus. This stable 
surface is characterized as having end space 
isomorphic to $(E,E^G)$, where $E$ and $E^G$ 
are Cantor sets. 
In this case, no ends are of finite orbit,
and hence there are no omnipresent arcs. However, 
any arc with one endpoint in $E$ and one 
endpoint in $E^G$ is a grand arc.
This shows that the grand arc graph can be defined on surfaces where the omnipresent arc graph is empty.

\subsubsection{Loxodromic actions of big mapping class groups}
In the literature, there has been a considerable amount of interest in studying loxodromic actions, weakly properly discontinuous (WPD) actions, or weakly WPD (WWPD) actions of big mapping class groups on various metric spaces \cite{AMP, Bav, HQR, MV, Rasmussen}.

If $G$ is a group acting on a $\delta$-hyperbolic path-connected metric space by isometries, we say that an element $g\in G$ acts \textit{loxodromically} if for any $x\in X$, $d(x,gx)$ is uniformly bounded from below. For more information on isometries of $\delta$-hyperbolic spaces, we refer the reader to \cite{GH}.

In Section \ref{Sec:QuasiContinuous}, we will prove that there exist loxodromic actions of the mapping class group acting on the grand arc graph coming from pseudo-Anosov mapping classes on witnesses (see Section ~\ref{Sec:WitnessesOfGrandArc} for the definition of a witness).

\begin{theorem} \label{thm:lox-action}
Let $W$ be a witness, and let $\vphi$ 
be a pseudo-Anosov mapping class that fixes the boundary of $W$. Let $\bar\vphi\in \Map(\Sigma)$ be the homeomorphism fixing 
$W^c$ and acting as $\vphi$ on $W$. Then 
$\bar\vphi$ acts loxodromically on $\calG(\Sigma)$.
\end{theorem}

\subsection{Outline of paper}

In Section \ref{Sec:WitnessesOfGrandArc}, we give an overview of grand arcs, infinite-type surfaces, and the end space. In this section, we also show that when the grand splitting is finite, the grand arc graph admits witnesses. 
In Section \ref{Sec:DisjointWitnesses}, we prove that the grand splitting contains exactly two equivalence classes of maximal ends if and only if the grand arc graph admits two disjoint witnesses. In Section \ref{Sec:Connectivity}, we prove that the grand arc graph is an infinite-diameter, connected metric space when $|\calS(\Sigma)|\ge 2$.

In Section \ref{Sec:Hyperbolicity}, we prove that if the grand arc graph admits two disjoint witnesses, then it is not $\delta$-hyperbolic.
Additionally, in this section we use the machinery of \cite{FGM} to show that when the grand splitting contains at least three equivalence classes of maximal ends, it is
$\delta$-hyperbolic. These results imply Theorem ~\ref{thm:InfDiamHyperbolic}. Finally, in Section \ref{Sec:QuasiContinuous}, we prove Theorem \ref{thm:mcgaction} and Theorem \ref{thm:lox-action}.\\

\noindent \textbf{Acknowledgements.}
We would like to thank Mahan Mj, Kasra Rafi and Ferr\'{a}n Valdez for helpful conversations.
We would like to thank Tyrone Ghaswala for comments and questions on an earlier draft.
The first author was partially supported by an NSERC-PGSD Fellowship, a Queen Elizabeth II Scholarship, and a FAST Scholarship at the University of Toronto.
The second author was supported by the National Science Foundation under Grant No. DMS-1928930 while participating in
a program hosted by the Mathematical Sciences Research Institute in Berkeley, California, during the Fall 2020 semester. The second author was also partially supported by an NSERC-PDF Fellowship.

\section{Witnesses of the Grand Arc Graph}\label{Sec:WitnessesOfGrandArc}

This section is an introduction of  some concepts and terminology on infinite-type surfaces, with the goal of proving that the grand arc graph has witnesses in the sense of \cite{S}. 
To begin, we introduce some definitions which will be used throughout this paper. For more detail, we refer the reader to \cite{MR}.

\subsection{Surfaces, ends, and maximal ends}
In this paper, a \textit{surface} is 
a 2-dimensional connected orientable topological manifold. If the surface has finitely-generated fundamental group, then it is called \textit{finite-type}. Otherwise, it will be called \textit{infinite-type}. Throughout this paper, we will denote infinite type surfaces by $\Sigma$, finite type surfaces by $S$, and witnesses by $W$.

The \textit{space of ends} of a surface $\Sigma$, denoted $E = E(\Sigma)$, is defined as 
$E(\Sigma) = \varprojlim_{K\subset \Sigma} \pi_0(\Sigma\ssm K)$, 
where $K$ ranges over all compact subsets of $\Sigma$, 
and if $K_i\subset K_j$, then the map in the projective 
limit is the inclusion 
$\vphi_{ij}:\pi_0(\Sigma\ssm K_j)\hookrightarrow \pi_0(\Sigma\ssm K_i)$. 
Note that when $\pi_0(\Sigma\ssm K)$ is endowed with the discrete topology, $E(\Sigma)$ becomes a totally disconnected, separable metrizable topological space. The 
space $\bar \Sigma = \Sigma \sqcup E(\Sigma)$ is called 
the \textit{end compactification} of $\Sigma$ \cite{Fr31}.
More concretely, an \textit{end} of $\Sigma$ can also be thought of as a sequence of nested open connected sets $\{U_n\}_{n=1}^{\infty}$ which are eventually disjoint from 
any compact $K\subset \Sigma$. For more detail, we refer the reader to \cite{FGM}. 

Let $\Sigma$ and $\Sigma'$ be two infinite type surfaces. Note that any continuous map $f:\Sigma' \to \Sigma$, induces a continuous 
map $f^*:E(\Sigma) \to E(\Sigma')$ because the inverse limit functor is contravariant. If $f:\Sigma' \hookrightarrow \Sigma$ is an embedding, then $f^*$ is surjective. 
Let $i:\Sigma'\hookrightarrow \Sigma$ be an inclusion. If an end of $\Sigma$ has precisely one preimage under $i^*$, we say that it \textit{lies in} the subsurface $\Sigma'$. 

An end $e$ is said to be \textit{accumulated by genus} if 
every subsurface that it lies in cannot be embedded in $\RR^2$. Otherwise, 
it is called \textit{planar}. We denote the set of ends accumulated by genus by $E^G(\Sigma)$, and note that it is a closed subset of $E(\Sigma)$. By Richards' classification of infinite-type surfaces \cite{R}, infinite-type 
surfaces are completely characterized by the triple $(E,E^G,g)$, 
where $0\le g\le \infty$ is the \textit{genus} of the surface.

In light of this, we think of a subset $U\subset E$ as the pair $(U,U\cap E^G)$, 
and when we say $U \subset E(\Sigma_1)$, and $V \subset E(\Sigma_2)$ are homeomorphic, we mean that there is a homeomorphism $\vphi:U\to V$ such that $\vphi|_{U\cap E^G(\Sigma_1)}$ is a homeomorphism  from $U\cap E^G(\Sigma_1)$ to $V \cap E^G(\Sigma_2)$. A set $U\subset E$ is \textit{self-similar} if for any decomposition of $U$ into two sets, $U_1$ and $U_2$, which are clopen in the subspace topology of $U$, we have that $U$ is homeomorphic to $U_1$ or to $U_2$. We say that a subsurface $S\subset \Sigma$ 
\textit{separates} two ends, $e,e' \in E(\Sigma)$, 
if $e$ and $e'$ lie in different connected components 
of $\Sigma\ssm S$.

In \cite{MR}, Mann--Rafi define a \textit{partial ordering on the space of ends} by setting $x \preccurlyeq y$, for $x,y \in E(\Sigma)$, if every clopen  neighbourhood of $y$ 
contains a clopen set homeomorphic to a neighbourhood of $x$.
Two ends are \textit{equivalent}, or \textit{are of the same type} if $x \preccurlyeq y$ and $y \preccurlyeq x$. 
This partial order has maximal elements, and we let 
$\calM(\Sigma)$ denote the space of maximal ends of $\Sigma$. A subset $X$ of a group $G$ is \textit{coarsely bounded}, abbreviated CB, if every compatible left-invariant metric on $G$ gives $X$ finite diameter.
Mann--Rafi show that for infinite-type surfaces whose mapping class group is locally CB, one can write the space of maximal ends of $\Sigma$ as a disjoint union of equivalence classes of self-similar ends. 

Let $\Sigma$ be an infinite-type surface 
which has finitely many equivalence 
classes of maximal ends. 
In particular, this means we can write 
\begin{align*}
\calM(\Sigma) = \sqcup_{i=1}^{n} E_i
\end{align*}
where the $E_i$'s are self-similar, meaning that they are either Cantor sets or singletons, and all of the ends $e \in E_i$ are equivalent in the sense above.
We call such a splitting of $\calM(\Sigma)$ a 
\textit{grand splitting}, denoted by $\calS(\Sigma)$, and note that it is unique. Thus, if $\Map(\Sigma)$ is locally CB, then $\Sigma$ has a finite grand splitting.
We think of the grand splitting, 
$\calS(\Sigma)$, of 
$\calM(\Sigma)$ as a 
collection of subsets of ends, whereas we think of $\calM(\Sigma)$ as just a set of ends.

\subsection{Arcs, Loops, and Witnesses}

Throughout this paper, an \textit{arc} is a proper map $\gamma:(0,1)\to \Sigma$ such that $\gamma$ extends continuously to a map $\bar\gamma:[0,1]\to \bar\Sigma$, where $\bar\gamma(0), \bar\gamma(1)\in E(\Sigma)$.
We call the points $\bar\gamma(0)$ and $\bar\gamma(1)$ the 
\textit{endpoints} of $\gamma$.
A \textit{curve} in $\Sigma$ is a map from $S^1$ to $\Sigma$.
Throughout this paper, we often conflate curves and arcs with their images in the surface.
A curve or arc is called \textit{simple} if 
it does not intersect itself.

On a finite-type surface, arcs behave nicely, for example, they intersect one another at most finitely many times. 
On infinite-type surfaces, this is not the case, as arcs can intersect infinitely often and become quite complicated.
Because of this, we restrict ourselves to focus on arcs which converge to two maximal ends.

\begin{definition}
An arc $\alpha$ \textit{converges} to an end $e$ if for any neighborhood $U$ of $e$, $\alpha$ eventually never leaves this neighborhood. For an infinite-type surface $\Sigma$, an arc $\alpha$ is called \textit{grand} if it is simple and converges to exactly two maximal ends, each of which is in a different element of $\calS(\Sigma)$. 
\end{definition}

In other words, grand arcs are two-ended 
arcs between maximal ends of different type, or of the same, finite-orbit type.
Using the definition of a grand arc, we are able to define the grand arc graph.

\begin{definition}
Let $\Sigma$ be an infinite-type surface. We denote the set of all grand arcs on $\Sigma$ by $\calG(\Sigma)$. The \textit{grand arc graph}, denoted by $\calG(\Sigma)$ is the combinatorial graph given by the following data:
\smallskip
\begin{center}
\begin{tabular}{p{2cm} p{9cm}}
Vertices & There is one vertex for every isotopy class of grand arcs in $\Sigma$.\\
Edges &There is an edge between any two vertices corresponding to isotopy classes $\alpha$ and $\beta$ of $\calG(\Sigma)$ if $\alpha$ and $\beta$ have disjoint representatives.
\end{tabular}
\end{center}

We equip $\calG(\Sigma)$ with a metric by setting the length of each edge to $1$.
This gives the grand arc graph the structure of a complete metric space, and endows it with the discrete topology.
\end{definition}

\begin{remark}
Notice that for the once-punctured Cantor tree, the grand arc graph is precisely equal to the ray graph defined by Bavard \cite{Bav}.       
\end{remark}

Because mapping classes preserve disjointness of arcs, it follows trivially that the mapping class group acts by isometries on $\calG(\Sigma)$.

The following two definitions are inspired by the works of Schleimer and Masur--Schleimer on witnesses and the curve complex for finite-type surfaces \cite{MS, S}.
 In these works, arcs can be studied by their projections to subsurfaces. In our case, the relevant subsurfaces are finite-type subsurfaces of an infinite-type surface $\Sigma$. We would like these projections to be nonempty, prompting the following definition:

\begin{definition}
Let $\Sigma$ be an infinite-type surface, 
$S\subset \Sigma$ a subsurface, and $\alpha$ some arc in $\Sigma$. 
\begin{itemize}
    \item We say that $S$ is \textit{essential} 
    if it is not homotopic to a point, a puncture, or a boundary component of $\Sigma$.
    \item We say that $S$ \textit{sees} $\alpha$ if every subsurface homotopic to $S$ intersects every arc homotopic to $\alpha$.
    \item We say that $S$ is a \textit{witness} 
    of $\calG(\Sigma)$ if it is essential, and if it sees every arc in $\calG(\Sigma)$. 
    \end{itemize}
    If there exists a surface $W \subset \Sigma$ which is a witness for $\calG(\Sigma)$, then we say that $\calG(\Sigma)$ has a witness, or, that $\calG(\Sigma)$ admits a witness.
\end{definition}

\begin{example}
The once-punctured Cantor tree of genus one admits two disjoint witnesses (Figure 1). Each is a witness, as every grand arc must converge to the puncture and to an end in the Cantor tree, and hence must intersect the surfaces that separate these sets of ends.
\begin{figure}[ht]
\begin{picture}(100,115)
\put(-25,0){\includegraphics[width=150\unitlength]{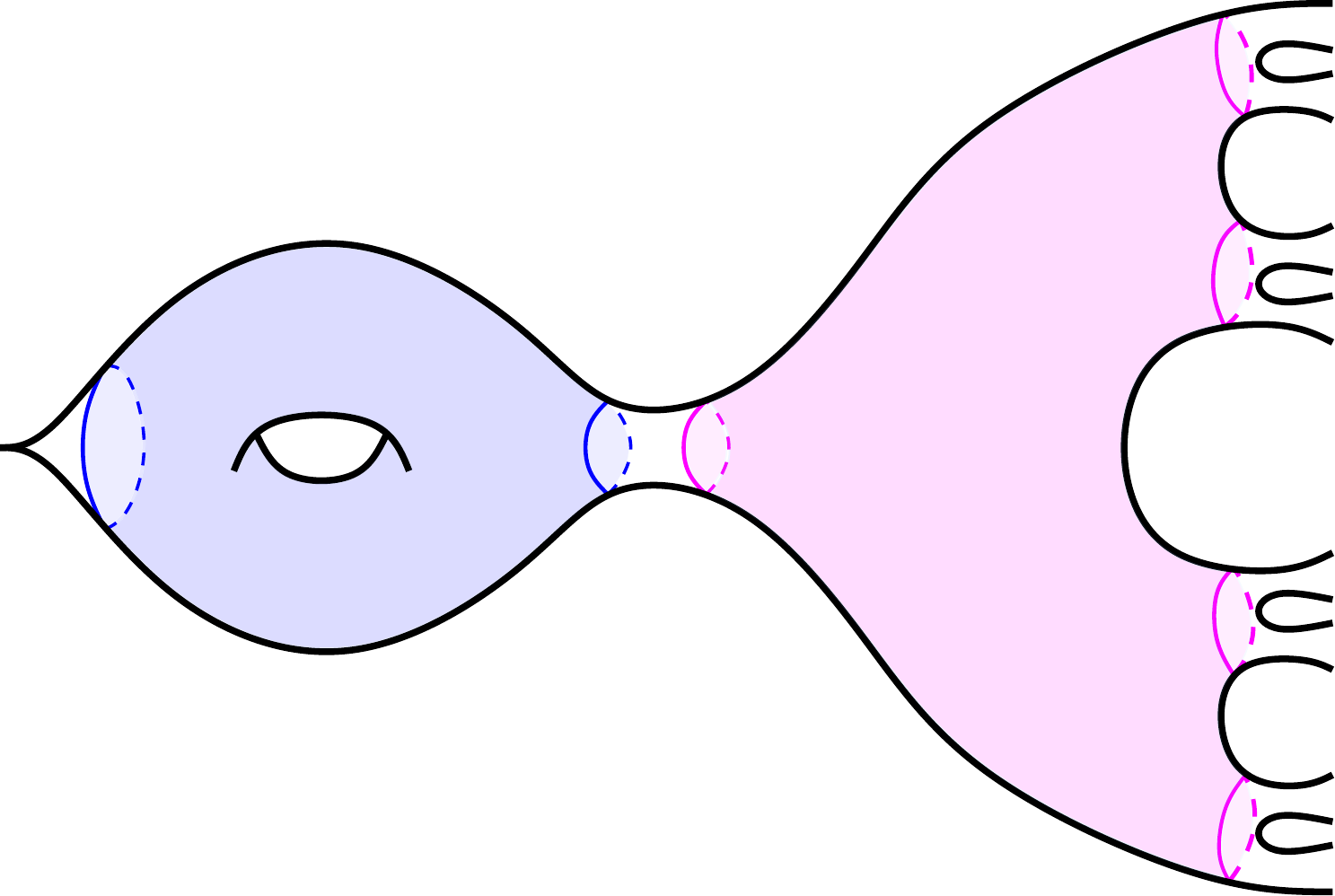}}
\end{picture}
\caption{Two disjoint witnesses on the once-punctured Cantor tree of genus one (cycloctopus).}
\label{Fig:CyclopotopusWitnesses} 
\end{figure}
\end{example}

In this paper, we will consider witnesses up to homotopy. In particular, when we say \textit{disjoint} witnesses, we mean two witnesses that are not homotopically-equivalent, and that have disjoint representatives in their respective homotopy classes. Since witnesses see all grand arcs, it follows that they are necessarily separating. We next introduce some useful notation to extend the definition of ``separating'' to the end space.

\begin{definition}
Let $S\subset \Sigma$ be a finite-type surface of $\Sigma$. If $E,E' \subset E(\Sigma)$, we say that $S$ \textit{separates} $E$ from $E'$ if the ends in $E_1$ lie in different connected components of $S^c$ from those in which the ends in $E_2$ lie.
\end{definition}

Using this definition, we see that a witness separates self-similar equivalence classes of maximal ends. We are now ready to prove that the grand arc graph has at least one witness. 

\begin{lemma} \label{lem:witness-exists}
Let $\Sigma$ be a surface of infinite type 
with $|\calS(\Sigma)| < \infty$. 
If $\calG(\Sigma)$ is nonempty, then it admits a witness. Moreover, this witness can be chosen to be homeomorphic to the $|\calS(\Sigma)|$-times punctured sphere.
\end{lemma}

\begin{remark}
If we assume that $\Sigma$ is a surface whose mapping class group is CB-generated, then this 
follows immediately from Proposition 5.4 
of \cite{MR}
\end{remark}
\begin{proof}
Let $\calM(\Sigma) = E_1 \sqcup \cdots \sqcup E_n$ 
be the grand splitting of $E = E(\Sigma)$, 
and let $U_1,\hdots, U_n$ be a disjoint clopen 
cover of $E$ such that $E_i \subset U_i$. 

To see why the cover $U_1, \ldots, U_n$ exists, we first note that due to Proposition 4.7 of \cite{MR}, the space of maximal ends is closed and the equivalence class of any maximal end is either finite or a Cantor set.
The end space is Hausdorff and compact, so it is a normal Hausdorff space.
Therefore, we can find a finite collection of disjoint open sets, $\{V_i\}_{i=1}^n$, such that $E_i\subset V_i$. 
Since $E$ is a totally disconnected metrizable compact space, $E$ has a 
basis of clopen sets. Thus, 
for every $i$ and every $e\in E_i$, 
there exists a clopen neighbourhood $e\in V_e \subset V_i$. By compactness of $E_i$, a finite union of 
the neighborhoods $V_e$ cover $E_i$ and is contained in $V_i$. Define $U_i'$ to be this finite union. 
Thus, we have found clopen sets 
$U_1',\hdots, U_n'$ such that $E_i\subset U_i'$. 
Setting 
$U_1 = U_1' \cup (E \ssm \cup_i U_i)$ 
and $U_i = U_i'$ for $i\ge 2$ gives us the 
desired cover.

Let $U_i^G$ denote the ends of $\Sigma$ accumulated 
by genus that are in $U_i$. By Richard's 
classification \cite{R}, there exist 
surfaces $\Sigma_i$, each with a single 
boundary component $\eta_i$, whose end spaces 
are homeomorphic to $(U_i,U_i^G)$. If the genus of $\Sigma$ is finite and equal 
to $g$, we redefine $\Sigma_1$ to be the 
connect sum of $\Sigma_1$ and a genus $g$ surface.

Let $W$ be the sphere with $n$ connected boundary 
components, labelled $\gamma_1,\ldots,\gamma_n$. By construction, and through Richard's 
classification, $\Sigma$ is homeomorphic to $S\cup_{\eta_i\sim\gamma_i} \Sigma_i$, and the 
image of $W$ under any homeomorphism is a witness.
\end{proof}

\section{Disjoint Witnesses and the Grand Splitting}\label{Sec:DisjointWitnesses}

The goal of this section is to prove a lemma which describes when the grand arc graph contains two disjoint witnesses.
This lemma will help prove when the grand arc graph is $\delta$-hyperbolic.
For the remainder of this paper, let $\Sigma$ be an infinite-type surface with a finite grand splitting, and let $E$ be its end space.

To begin, we state a lemma which shows that there exists a simple closed curve which separates any two disjoint clopen neighborhoods in the end space.
This lemma is a restatement of Lemma 2.1 from \cite{FGM}.

\begin{lemma} \label{Lemma:SCCSeparatesEnds}
Let $E_1,E_2$ be two disjoint clopen neighbourhoods of the endspace, $E = E(\Sigma)$. 
Then there exists a simple closed curve $\gamma$ in $\Sigma$ such that $\gamma$ separates the ends in $E_1$ from those in $E_2$.
\end{lemma}

\begin{remark}
If $E_i$ consists of a single planar end
and $\Sigma$ has genus, we can choose $\gamma$ to be essential by shifting one genus to $\Sigma_i$. Furthermore, if $\Sigma$ has at least $g$ genus, we can choose $g$ simple closed curves which are pairwise disjoint, and which separate $E_1$ from $E_2$.
\end{remark}

\subsection{Disjoint Witnesses and The Endspace}

In this section, we prove that the grand splitting contains exactly two equivalence classes of maximal ends if and only if the grand arc graph admits at least two disjoint witnesses.

\begin{lemma}\label{Lem:TwoEquivalenceClasses}
Let $\calG(\Sigma)$ be nonempty and let $|\calS(\Sigma)|<\infty$. Suppose that $\Sigma$ is neither the once-punctured Cantor tree, nor the cycloctopus (the once-punctured, genus 1 Cantor tree). Then the following are equivalent:
\begin{itemize}
    \item[(1)] $|\calS(\Sigma)| = 2$.
    \item[(2)] $\calG(\Sigma)$ has at least two disjoint witnesses.
    \item[(3)] $\calG(\Sigma)$ has at least two disjoint witnesses that are annuli.
    \item[(4)] $\calG(\Sigma)$ has at least two disjoint witnesses, $W_1$ and $W_2$, such that $\chi(W_i) \leq -2$.
\end{itemize}
\end{lemma}

\begin{proof}
 \ \vspace{0.1cm}

\noindent \textit{1 $\implies$ 3:}
Let $\calM(\Sigma) = E_1\sqcup E_2$ be the grand splitting of $\Sigma$. Note that if the genus of $\Sigma$ is at least $2$, we can use the remark following Lemma ~\ref{Lemma:SCCSeparatesEnds} to prove the claim. Thus we assume that the genus of $\Sigma$ is at most $1$. 
Since $\Sigma$ is not the cycloctopus, it follows that if the genus of $\Sigma$ is at most $1$, 
then there must be non-maximal predecessor ends to either the ends in $E_1$ or the ends in $E_2$.

We write $E = E_1\sqcup E_2 \sqcup U$, 
where $U = (E_1\cup E_2)^c$.
Note that $U$ must have an infinite number of ends, else $U$ would contain a maximal end itself, a contradiction.
Let $F_1$ and $F_2$ be disjoint clopen
subsets of $U$. We can choose the sets $F_1$ and $F_2$ because $U$ is an infinite open subset of a separable totally disconnected, compact topological space.

Using the classification of infinite-type surfaces \cite{R}, we construct three surfaces as follows; let:

\begin{itemize}
\item $\Sigma_1$ be the surface of genus $0$ with end space equal to $E_1\sqcup F_1$ and one boundary component, $\gamma_1$.
\item $\Sigma_2$ be the genus $g(\Sigma)$ surface with end space equal to $U\ssm (F_1\cup F_2)$, and two boundary components, $\eta_1$ and $\eta_2$.
\item $\Sigma_3$ be the surface of genus $0$ with end space equal to $E_2 \sqcup F_2$, and one boundary component, $\gamma_2$.
\end{itemize}

Gluing together these surfaces along 
$\gamma_i \sim \eta_i$ yields a surface $\Sigma'$ 
with end space isomorphic to $E$, 
and genus equal to the genus of $\Sigma$. 
By the Classification of Infinite-Type 
Surfaces, $\Sigma' = \Sigma$.
The image of $\gamma_i \sim \eta_i$ in $\Sigma'$ 
give core curves of two disjoint, essential witnesses that are annuli.

\noindent \textit{3 $\implies$ 2:}
This direction follows trivially.

\noindent \textit{2 $\implies$ 1:}
Suppose that $\calG(\Sigma)$ has at least two 
witnesses, $W_0$ and $W_1$. Let $\Sigma_i$ 
be the connected component of $\Sigma \ssm W_i$ 
which contains $W_{1-i}$. If $W_i$ is a witness, 
then $\Sigma_i$ must contain ends from exactly 
one subset $E_i \in \calS(\Sigma)$. Noting 
that $\Sigma = \Sigma_1\cup \Sigma_2$ and that 
$\Sigma_1\cap \Sigma_2$ cannot contain any maximal 
ends yields that $|\calS(\Sigma)| = 2$.

\noindent \textit{4 $\implies$ 2:}
This direction follows trivially.

\noindent \textit{3 $\implies$ 4:}
Let $\gamma$ be a witness which is an annulus.
Then $\gamma$ separates $\Sigma$ 
into two connected components $\Sigma_1$ and $\Sigma_2$. 
We show that $\Sigma_1$ contains a witness $S$ 
with $\chi(S) \le -2$. 
We note that any subsurface $S$ containing $\gamma$ must 
be a witness.

If $\Sigma_1$ is of finite-type, then the only cases 
where $\chi(\Sigma_1) > -2$ are when $\Sigma_1$ 
is an annulus or a pair-of-pants. The former case 
is excluded by $\gamma$ being essential. In the latter 
case, the two boundary components of $\Sigma_1$ 
which are not $\gamma$ must be isolated punctures, meaning that $\gamma$ is 
not a witness, a contradiction.

If $\Sigma_1$ is of infinite-type, then any compact 
exhaustion of $\Sigma_1$ which contains the boundary 
component $\gamma$ must eventually be the desired witness.

Applying the same argument to $\Sigma_2$ yields 
two disjoint witnesses of sufficient complexity, as desired.
\end{proof}

\begin{remark}
    If $\Sigma$ is the cycloctopus (once-punctured Cantor tree of genus $1$), we have 
    $|\calS(\Sigma)| = 2$, and 
    $\calG(\Sigma)$ has two disjoint 
    essential witnesses. See Figure \ref{Fig:CyclopotopusWitnesses}. However, we cannot choose both of these witnesses to be annuli. 
\end{remark}

\section{Connectivity of the grand arc graph}\label{Sec:Connectivity}

In this section, we show that the grand arc graph has infinite diameter. 
In addition, we show that for surfaces which have at least three elements in the grand splitting, the grand arc graph is connected.

To show that the grand arc graph has infinite diameter, we will use subsurface 
projections, following the ideas of Schleimer \cite{S}, 
to project the grand arc graph to the arc graph of a witness. To show that $\calG(\Sigma)$ is connected, we will use unicorn paths as in \cite{HPW}, but extended to the setting of infinite-type surfaces as in \cite{FGM}.

To begin, we prove the following lemma which shows that the subsurface projection of a grand arc to a finite-type subsurface ${S\subset \Sigma}$ will have finitely many connected components.
In the statement of the lemma, below, we let $\pi_S(\alpha)$ denote the 
subsurface projection of $\alpha$, as defined in  \cite{MM2}.
Namely, we define the projection map, $\pi_S(\alpha)$, from the grand arc graph, $\calG(\Sigma)$, to the arc graph, $\calA(S)$, by setting $\pi_S(\alpha)$ to be the multi-arc $S\cap \alpha$.

\begin{lemma}
Let $\alpha\in \calG(\Sigma)$ be a grand arc, and let $S\subset\Sigma$ be a finite-type 
subsurface whose boundary is in minimal 
position with $\alpha$. 
Then $\pi_S(\alpha)$ has finitely many 
(possibly zero) connected components.
\end{lemma}

\begin{remark}
The condition on $\del S$ can be automatically satisfied 
for any grand arc if we take a homotopy class of $S$ with geodesic boundary with respect to some hyperbolic metric on $\Sigma$.
\end{remark}

\begin{proof}
The boundary of $S$ is a finite collection of 
simple closed curves on $\Sigma$ in minimal 
position with respect to $\alpha$. Reparametrizing the image of $\alpha$ as  $\alpha:\RR \to \Sigma$, We find that there exists 
an $M > 0$ such that the restriction of $\alpha$ to $[-M,M]^c$, $\alpha_{[-M,M]^c}$, is disjoint from 
$\del S$ (this is because $\alpha$ is a grand arc). Note 
that the image of $[-M,M]$ under $\alpha$ is a compact 
arc segment, so there exists a finite-type subsurface
$K \subset \Sigma$ containing $S, \del S$ and $\alpha_{[-M,M]}$. 
The lemma follows by noticing that $\alpha$ and 
$\del S$ can be realized as geodesics (in some  hyperbolic metric on $\Sigma$) in minimal position on the finite type surface $K$.
\end{proof}

For a given finite-type surface $S\subset \Sigma$, let $\calA(S)$ be the arc complex of $S$, and let $\alpha$ and $\beta$ be two multi-arcs.
Then, as in \cite{MT}, we may define the distance between two multi-arcs by:
\begin{align*}
d_{\calA(S)}(\pi_S(\alpha),\pi_S(\beta)) 
= \inf \{d_{\calA(S)}(a,b) : a\in \pi_S(\alpha), b\in \pi_S(\beta) \}
\end{align*}
When $\alpha$ and $\beta$ are grand arcs, and $W$ is a witness, the above infimum is realized as a minimum.

Let $W$ be a witness of the grand arc graph.
The next lemma shows that the distance of two grand arcs in the grand arc graph is bounded below by the distance of their subsurface projections to the arc graph of the witness, $\calA(W)$.

\begin{lemma}\label{Lem:SubsurfaceProjections}
Let $W\subset \Sigma$ be a witness of $\calG(\Sigma)$. Let $\alpha, \beta \in \calG(\Sigma)$ be grand arcs which are in minimal position 
with respect to $\del W$. Then
\begin{align*}
d_{\calA(W)}(\pi_W(\alpha),\pi_W(\beta)) \le
d_{\calG(\Sigma)}(\alpha,\beta).
\end{align*}
\end{lemma}

\begin{proof}
Let $\alpha = \alpha_1,\hdots,\alpha_n = \beta$ be a path in 
$\calG(\Sigma)$. Note that since $W$ is a witness, $\pi_S(\alpha_i)$ must be nonempty. 
For each $i$, choose an arc $a_i\in \pi_S(\alpha_i)$.
Since $\alpha_1, \ldots, \alpha_n$ is a path between $\alpha$ and $\beta$ in the grand arc graph, $a_1, \ldots, a_n$ is a path between $\pi_{W}(\alpha)$ and $\pi_{W}(\beta)$ in $\calA(W)$. The claim follows. 
\end{proof}

Since we are able to define projections from $\calG(\Sigma)$ to $\calA(W)$, where $W$ is a witness of $\calG(\Sigma)$, we can ask whether $\calG(\Sigma)$ is 
quasi-isometric to the subcomplex of $\calA(W)$ where we only consider the arcs which start 
and end at the boundary components which separate maximal ends.
While the general case is beyond the scope of this paper, there are examples where this is not the 
case, namely, where $|\calS(\Sigma)| = 3$, and 
$W$ is a pair of pants. In this case, 
$\calA(W)$ has finite diameter, contradicting the following lemma:

\begin{lemma}\label{Lem:GrandArcInfiniteDiameter}
If $1 < |\calS(\Sigma)| < \infty$, then $\calG(\Sigma)$ has infinite diameter.
\end{lemma}

\begin{proof}
Let $W$ be a witness of $\Sigma$. After possibly extending $W$, we can assume that $W$ is a connected subsurface of $\Sigma$ for which $\calA(W)$ has infinite diameter \cite{MS}.

Let $a \in \calA(W)$ and let $C_1, C_2$ be two boundary components or punctures of $W$.
First, we will show that there exists an arc 
$a'' \in \calA(W)$ whose endpoints lie on $C_1$ and $C_2$, 
and $d_{\calA(W)}(a,a'') \le 2$. Fix a hyperbolic metric $d_W$ on $\Sigma$, and realize $W$ as a surface with geodesic boundary. Let $a_1$ and $a_2$ be the endpoints of the arc $a$, and let $x \in a$ and $y \in C_1$ be the points for which
\[
d_W(x,y) = \min\limits_{t\in a, s\in C_1} d_{W}(t,s).
\]
Notice, we can write $a = [a_1,x] \cup (x, a_2]$.
Let $[x,y]$ denote the geodesic segment from $x$ to $y$, and let 
$a' = [x,y] \cup (x,a_2]$. This arc has endpoints $a_2$ and $x\in C_1$, and is 
disjoint from $a$. Repeating this trick, we can get an arc $a''$ 
disjoint from $a'$ with endpoints in $C_1$ and $C_2$.

Fix any $n \in \NN$, and choose $a,b\in \calA(W)$ so that 
$d_{\calA(W)}(a,b) = n+4$. Let $C_1$ (resp. $C_2$), either be a maximal end $e_1$ (resp. $e_2$), contained in $S$, or let $C_1$ (resp. $C_2$) be a boundary component of $W$ which separates $\mathrm{int}(S)$ from a maximal end $e_1$ (resp. $e_2$). 
By the triangle inequality and the above paragraph, there exist arc segments $a'$ and $b'$ with
endpoints on $C_1$ and $C_2$ for which 
$d_{\calA(W)}(a',b') \ge n$. Extend the arc segments $a'$ and $b'$ to be grand arcs by 
adding arc segments from their endpoints to the ends $e_1$ and $e_2$. Call 
these new arcs $\alpha$ and $\beta$.
By construction, $\pi_W(\alpha) = a$ and $\pi_S(\beta) = b$, which implies by Lemma ~\ref{Lem:SubsurfaceProjections} that $d_{\calG(\Sigma)}(\alpha,\beta) \ge n$.
\end{proof}

Using Lemma \ref{Lem:GrandArcInfiniteDiameter}, we are now in a position where we can show that the grand arc graph is an infinite-diameter, connected metric space.\\

\begin{proposition} ~\label{prop:InfDiam}
If $|\calS(\Sigma)| \ge 3$, then 
$\calG(\Sigma)$ is an infinite-diameter, 
connected metric space.
\end{proposition}

To prove this, we will use \textit{unicorn paths} as introduced by Hensel--Przytycki--Webb \cite{HPW}.
Unicorn paths were originally introduced for arcs which have a finite number of intersections
and were extended to the setting of infinitely many intersections by Fanoni--Ghaswala--McLeay \cite{FGM}.

A \textit{unicorn} in $\alpha$ and $\beta$ is an arc of the form $a \cup b$, where $a$ is the closure of a connected component
of $\alpha \setminus p$, $b$ is a connected component of $\beta \setminus p$, and $p$ is an intersection point of $\alpha$ and $\beta$. We call the
point $p$ the \textit{corner} of the unicorn.
Note that for $a \cup b$ to be an arc, $a$ and $b$ can intersect only at $p$. In general, given two subarcs of $\alpha$
and $\beta$ with a common endpoint in $\Sigma$, there can be other intersection points, in which case the subarcs do not define a unicorn.

\begin{proof}
By Lemma \ref{Lem:GrandArcInfiniteDiameter}, we have that $\calG(\Sigma)$ has infinite diameter when $\abs{\calS(\Sigma)} \geq 2$.  Therefore, we need only show that $\calG(\Sigma)$ is a connected metric space. To do this, we show that between any two grand arcs $\alpha$ and $\beta$, there exists a path in $\calG(\Sigma)$ between $\alpha$ and $\beta$.

Let $\alpha$ and $\beta$ be grand arcs.
There are three main cases to consider: 
\begin{itemize}
\item[(1)] The endpoints of $\alpha$ and $\beta$ all converge to different elements of the grand splitting.
\item[(2)] One endpoint of $\alpha$ converges to the same element of the grand splitting as one endpoint of $\beta$. 
\item[(3)] The endpoints of $\alpha$ and $\beta$ 
converge to the same elements of the grand splitting.
\end{itemize}

In the first case, if the endpoints of $\alpha$ 
and of $\beta$ are disjoint, then $\alpha$ 
and $\beta$ must intersect finitely many times.
This implies that there exists a finite unicorn path between 
them. Note we can choose this unicorn path to be a path of grand arcs.

In the second case, let $e_1,e_2$ be the 
endpoints of $\alpha$, and let $e_1',e_3$ 
be the endpoints of $\beta$. By assumption, there are 
three elements of the grand splitting, $E_1,E_2,E_3$ 
such that $e_1$ and $e_1'$ converge to $E_1$, $e_2$ converges to $E_2$, and $e_3$ converges to $E_3$.
Let $\gamma_1$ be the arc obtained by 
fellow-travelling along $\alpha$ from $e_2$ 
to its first intersection point $p$ with $\beta$, 
then fellow-travelling along $\beta$ from $p$ converging to $e_3$. Note that $\gamma_1$ is disjoint from
$\beta$ by construction. Moreover,
$\gamma_1$ is a grand arc intersecting $\alpha$ 
finitely many times. We orient $\alpha$ 
to start at $e_1$ and end at $e_2$, and orient 
$\gamma_1$ to start at $e_3$ and end at $e_2$. Note 
that since $|\gamma_1\cap \alpha| < \infty$, 
there exists a unicorn path between these two oriented 
arcs, which, together with $\gamma_1$, gives a 
path from $\alpha$ to $\beta$ in $\calG(\Sigma)$.

In the third case, let $e_1$ and $e_2$ be the 
endpoints of $\alpha$ and let $e_1'$ and 
$e_2'$ be the endpoints of $\beta$.
By assumption, we know that there are two elements of the grand splitting, $E_1, E_2 \in \calS(\Sigma)$, such that $e_1$ and $e_1'$ converge to $E_1$ and $e_2$ and $e_2'$ converge to $E_2$.
We will show that the third case reduces to the second case.
Indeed, notice that only two of the at least three equivalence classes of ends have $\alpha$ and $\beta$ converging towards them.
Choose one of the remaining equivalence classes of ends, $E_3$, and choose an arc $\gamma '$ which converges to the ends $e_1 \in E_1$ and $e_3 \in E_3$. 
If 
$\gamma'$ is disjoint from $\alpha$, then we are done.
If $\gamma'$ is not disjoint from $\alpha$, we set 
$\gamma$ to be the unicorn obtained from 
$\alpha$ and $\gamma'$ by fellow-travelling along 
$\gamma$ until the first point of intersection $p$
with $\alpha$, and then fellow-travelling from along $\alpha$ from $p$ and converging 
towards $e_1$.
By construction, $\gamma$ and $\alpha$ are disjoint grand arcs so that only one end point of $\alpha$ and one endpoint of $\beta$ converge towards the same element of the grand splitting, which is precisely Case (2) and we have proven our claim.
\end{proof}

\section{Hyperbolicity of the grand arc graph}\label{Sec:Hyperbolicity}

In this section, we provide a full characterization of when the grand arc graph is or is not hyperbolic. This, together with Proposition ~\ref{prop:InfDiam} will prove Theorem ~\ref{thm:InfDiamHyperbolic}.

\subsection{Grand arc graphs which are not hyperbolic}

In this subsection, we prove that if a grand arc graph of a surface $\Sigma$ admits two disjoint witnesses, then the grand arc graph is not $\delta$-hyperbolic.
To prove this result, we show that there exists a quasi-isometric embedding of $\mathbb{Z}^2$ into the grand arc graph for such a surface.
This is an analogue of Schleimer's disjoint witnesses principle, which holds for surfaces of finite-type \cite{S}. A shorter proof of this proposition can be found at the end of Subsection 6.2.

\begin{proposition}\label{prop:TwoWitnessesNotHyperbolic}
Suppose that $|\calS(\Sigma)| < \infty$.
If $\calG(\Sigma)$ is nonempty, and admits two disjoint witnesses $W_1,$ and $W_2$ with complexity $\chi(W_i)\le -2$, then $\calG(\Sigma)$ is not $\delta$-hyperbolic.
\end{proposition}

\begin{proof}
Since $W_1$ and $W_2$ are both witnesses, they must each 
have at least two boundary components. 
Let $\eta$ be a grand arc which intersects exactly two of the boundary components of $\del W_1$ exactly once, 
and exactly two of the boundary components of $\del W_2$ exactly once. Such an arc exists by taking simple arcs in $W_i$ with appropriate boundaries, connecting them in the finite-type connected component of ${\Sigma \ssm (W_1\cup W_2)}$, and extending the resulting arc segment to a grand arc.

The \textit{prescribed-boundary arc complex}, denoted by $\hat\calA(W_i)$, is defined to be the simplicial complex where the vertices are arcs 
in $W_i$ whose endpoints are $\eta \cap W_i$, and edges represent disjointness.
Using unicorn paths as in \cite{HPW}, it is evident that for any witness $W_i$, $\hat\calA(W_i)$ is connected. 
Since $d_{\hat\calA(W_i)} \ge d_{\calA(W_i)}$, it follows that it is of infinite diameter. 
Moreover, the pure mapping class group of $W_i$ acts by isometries on the metric space $\hat\calA(W_i)$. 

Let $\vphi_1$ and $\vphi_2$ be pseudo-Anosov 
mapping classes on $W_1$ and $W_2$, which 
exist since $\chi(W_i) \le -2$. 
Up to taking 
powers, we can assume that $\vphi_i$ fixes 
$\del W_i$ pointwise, and so, abusing notation, we 
can extend them continuously to mapping 
classes $\vphi_1$ and $\vphi_2$ in $\Map(\Sigma)$. 
Since the $\vphi_i$ are pseudo-Anosov, $d_{\hat \calA(W_i)}(\vphi_i^n\eta_i,\eta_i)$ 
tends to infinity as $n$ gets large \cite{MM1} \cite{MS}. 
Let $\eta^{m_1,m_2} = \vphi_1^{m_1}\vphi_2^{m_2}\eta$.

For any fixed $m_2$, we claim that the arcs 
$\{\pi_{W_1}(\eta^{j,m_2})\}_{j\in \ZZ}$ define a bi-infinite 
$(C_1,D_1)$-quasi-geodesic in $\hat\calA(W_1)$, where 
$C_1$ only depends on $W_1$ and $\vphi_1$.
Indeed, let $\eta_1 = \pi_{W_1}(\eta)$, and note that $\pi_{W_1}(\eta^{j, m_2}) = \vphi_1^{j}\circ \eta_1$. 
The claim follows since pseudo-Anosov 
maps act coarsely loxodromically on $\hat\calA(W_1)$
\cite{MS} \cite{MM1}.
Similarly, for any fixed $m_1$, the arcs $\{ \pi_{W_2}(\eta^{m_1,k})\}_{k \in \ZZ}$ define a bi-infnite $(C_2, D_2)$-quasi-geodesic in $\hat\calA(W_2)$, where 
$C_2$ only depends on $W_2$ and $\vphi_2$.
Using these bi-infinite quasi-geodesics and the same argument as in the proof of Lemma ~\ref{Lem:SubsurfaceProjections}, it follows that:
\begin{align*}
d_{\calG(\Sigma)}(\eta^{m_1,m_2}, \eta) 
&\geq \max \{d_{\hat\calA(W_1)}(\pi_{W_1}(\eta^{m_1,m_2}), \pi_{W_1}(\eta)), d_{\hat\calA(W_2)}(\pi_{W_2}(\eta^{m_1,m_2}), \pi_{W_2}(\eta))\}\\
&\geq C \max(|m_1|, |m_2|) - D
\end{align*}
where $C = \max\{C_1, C_2\}$ is the maximum of the two translation lengths of $\vphi_1$ 
or $\vphi_2$ acting on their respective arc complexes,  $\hat\calA(W_i)$, and $D = \max\{D_1,D_2\}$.

We will now translate these $(C_i,D_i)$-quasi-geodesics
in $\hat\calA(W_i)$ into a geodesic in $\calG(\Sigma)$.
We begin by choosing paths in $\hat\calA(W_i)$ 
between $\eta_i$ and $\vphi_i\eta_i$ of length $C_i$, 
and apply $\vphi_i$ to them to get a bi-infinite 
quasi-geodesic, $\Gamma_i$ in $\hat\calA(W_i)$.

Given any pair of arcs $\nu_1 \in \Gamma_1$ 
and $\nu_2 \in \Gamma_2$, we can construct a grand 
arc by taking 
$\nu_1 \cup \nu_2 \cup (\eta \cap (W_1\cup W_2)^c$. We 
call this constructed arc the \textit{lift} of the pair 
$(\nu_1,\nu_2)$. 

Let $F$ be the family of grand arcs obtained by 
lifting the family $\{(\nu_1,\nu_2): \nu_i\in \Gamma_i\}$. 
This family contains a path 
of length at most $C(|m_1| + |m_2|)+D$ between 
$\eta$ and $\eta^{m_1,m_2}$. 
Indeed, we can describe this path as follows.
Let $\Gamma_i(j)$ denote the $j^{\text{th}}$ point in the path $\Gamma_i$.
From above, we know that $\Gamma_1$ is a quasi-geodesic path of length at most $C_1 m_1 + D_1$, and $\Gamma_2$ is a quasi-geodesic path of length at most $C_2 m_2 + D_2$.
Therefore, the lifts of the sequence \[
\{(\Gamma_1(i),\Gamma_2(0))\}_{0\le i\le C_1 m_1 + D_1}
\]
is a path of length $C_1 m_1 + D_1$ of grand arcs from $\eta^{0,0}$ to $\eta^{m_1,0}$, and the lifts of the sequence \[
\{(\Gamma_1(C_1 m_1 + D_1),\Gamma_2(j))\}_{0\le j \le C_2 m_2 + D_2}
\] 
is a path of length $C_2 m_2 + D_2$ of grand arcs from $\eta^{m_1,0}$ to $\eta^{m_1,m_2}$. Concatenating the two paths gives us a path of length at most $C(|m_1| + |m_2|)+D$ between 
$\eta$ and $\eta^{m_1,m_2}$. 

Since $\Map(\Sigma)$ acts by isometries on $\calG(\Sigma)$, 
it follows that for any $m_1,m_2,m_1',m_2'$: 
\begin{align*}
\frac{1}{2C} (|m_1-m_1'| + |m_2-m_2'|)-D
&\le d_{\calG(\Sigma)}(\eta,\eta^{m_1-m_1',m_2-m_2'}) \\
&= d_{\calG(\Sigma)}(\eta^{m_1,m_2},\eta^{m_1',m_2'}) \\
&\le C(|m_1-m_1'| + |m_2-m_2'|)+D
\end{align*}
which implies that $\{\eta^{m_1, m_2}\}_{m_1, m_2\in \ZZ}$ 
forms a quasi-isometrically embedded copy of $\ZZ^2$ in 
$\calG(\Sigma)$. Thus, $\calG(\Sigma)$ 
is not $\delta$-hyperbolic.
\end{proof}

\begin{remark}
The above proposition can be generalized to 
cases when there are $k$ (or infinitely-many) disjoint witnesses of sufficient 
complexity, in which case 
$\calG(\Sigma)$ has a quasi-isometrically 
embedded copy of $\ZZ^k$ (or $\ZZ^{\NN}$). 

A similar result along these lines can be found in \cite{GRV}
where it is shown that for surfaces containing an \textit{essential shift},
one can quasi-isometrically embed an infinite dimensional cube into the mapping class group of the surface.
\end{remark}

\subsection{Grand arc graphs which are hyperbolic}

In this section, we prove when the grand arc graph of a surface $\Sigma$ is $\delta$-hyperbolic.
This section closely mirrors section 
6 of \cite{FGM}, which proves the hyperbolicity of the omnipresent arc graph.
The lemmas and corollaries in this section come from the work of Fanoni--Ghaswala--McLeay on omnipresent arcs, but the statements and proofs also hold in the case of grand arcs.

To begin, we define \textit{grand unicorn paths}, an extension of the definition of unicorn paths which were introduced in Section \ref{Sec:Connectivity}.
A \textit{grand unicorn} is a unicorn that is 
also a grand arc. If $\alpha$ and $\beta$ are grand 
arcs, we let $\calU_G(\alpha,\beta)$ denote
the set of all grand unicorns originating 
from $\alpha$ and $\beta$ and call this a \textit{grand unicorn path}.

The following lemma is the analogue of Lemma 6.5 in the work of Fanoni--Ghaswala--McLeay \cite{FGM}, but is based on Remark 3.2 of Hensel--Przytycki--Webb\cite{HPW}.

\begin{lemma}[Lemma 6.5 of \cite{FGM}]
\label{FGM6.5}
Let $x = a \cup b$ be a grand unicorn in $\calU_{G}(\alpha, \beta)$. Then one of the following holds:
\begin{itemize}
    \item  $\beta \ssm b$ intersects $a$, and then there exists a unicorn $y = a' \cup b' \in \calU_{G}(\alpha, \beta)$ which is disjoint from $x$, such that $a' \subsetneq a$ and $b \subsetneq b'$. In particular, the corner of $y$ belongs to the interior of $a$; or
    \item $\beta \ssm b$ and $a$ are disjoint arc segments. In this case, $x$ is disjoint from $\beta$.
\end{itemize}
\end{lemma}

A consequence of Lemma \ref{FGM6.5} is the following corollary, which is the analogue of Corollary 6.6 of \cite{FGM}.
This corollary shows that if two grand arcs intersect a finite number of times, then the subgraph $\calU_{G}(\alpha, \beta)$ of $\calG(\Sigma)$ is connected.

\begin{corollary}[Corollary 6.6 of \cite{FGM}]\label{FGM6.6}
Suppose that $x = a \cup b \in \calU_{G}(\alpha, \beta)$ and $\beta\ssm b$ has finite intersection with $a$,
Then there exists a unicorn path between
$x$ and $\beta$ in $\calU_{G}(\alpha, \beta)$.
\end{corollary}

In the above lemma, we only considered the case where grand arcs intersect a finite number of times.
Therefore, it does not follow that $\calU_{G}(\alpha, \beta)$ is a connected subgraph of $\calG(\Sigma)$ when $\alpha$ and $\beta$ intersect an infinite number of times.
The next lemma shows that so long as $\alpha$ and $\beta$ intersect, then $\calU_{G}(\alpha, \beta)$ must contain a grand arc which is neither $\alpha$ nor $\beta$.
While this lemma is stated in Fanoni--Ghaswala--McLeay in the case of omnipresent arcs, it extends to the case of grand arcs using the same reasoning.

\begin{lemma}[Lemma 6.7 of \cite{FGM}] \label{lem:exists_grand_unicorn}
Let $\alpha, \beta$ be two intersecting grand arcs. Then $\calU_G(\alpha, \beta)$ contains an arc $x$ which is not equal to $\alpha$ or $\beta$
\end{lemma}

As noted earlier, $\calU_G(\alpha, \beta)$ may not be a connected subgraph of $\calG(\Sigma)$.
However, the following lemma shows that the one-neighborhood of $\calU_G(\alpha, \beta)$ is a connected subgraph of $\calG(\Sigma)$.

\begin{lemma}\label{Lemma:FGM6.8}
Suppose $|\calS(\Sigma)| \geq 3$. For any $\alpha, \beta \in \calG(\Sigma)$, the full subgraph of $\calG(\Sigma)$ spanned by $N_1(\calU_G(\alpha, \beta))$
is connected.
\end{lemma}

The proof of this lemma is similar to the 
proof of Lemma 6.8 in \cite{FGM}, 
where we replace $P$ with $\calS(\Sigma)$, 
$\calA_2(\Sigma, P)$ with $\calG(\Sigma)$, 
$\calU_2(\alpha,\beta)$ with $\calU_G(\alpha,\beta)$, 
and the surface with a witness.

\begin{proof}
Assume that $|\alpha\cap \beta| > 0$, otherwise the result is clear. Fix a witness $W$.

Let $x_0 = a_0 \cup b_0 \in \calU_G(\alpha,\beta)$ 
be a unicorn arc which is not equal to either $\alpha$ 
or $\beta$, and let $p_0$ be the corner of $x_0$.
Let $b_0'$ be the remaining arc segment of 
$\beta$, such that $\beta = b_0 \cup b_0'$. 
If $|a_0\cap b_0'| < \infty$, then 
there is a finite unicorn path 
from $x_0$ to $\beta$ by Corollary 
~\ref{FGM6.6}. If not, then 
there exists a sequence of unicorn paths, 
$x_k = b_k\cup a_k$ such that 
$b_k \subsetneq b_{k+1}$, and whose 
corners eventually leave $W$. In 
particular, this means that for some $k$, 
$x_k \cap W = \beta \cap W$.

Let $E_1,\hdots, E_n = \calS(\Sigma)$ be the 
grand splitting of $\Sigma$. 
Since $W$ is a witness, for any $i$ there exists 
a finite multicurve, $C_i \subset \del W$ which separates $E_i$ from 
the other ends of $\Sigma$. Without 
loss of generality, let the ends of 
$\beta$ lie in $E_1$ and $E_2$.

Let $\gamma_S$ be a simple arc from a 
boundary component in $C_3$ to a point 
in $S\cap \beta$. Let $\gamma_3$ be an simple 
arc starting at $\gamma_S\cap C_3$ and converging 
to an end in $E_3$, and let $b$ be a 
subarc of $\beta$ starting at $\gamma_3 \cap \beta$ 
and converging to an end, such that $b$ is 
contained in $b_k$.
By construction, $\gamma_3 * \gamma_S * b$ 
is disjoint from $x_k$ and from $\beta$, 
proving the theorem.
\end{proof}

To prove that the grand arc graph is $\delta$-hyperbolic, we use a version of the ``guessing geodesics lemma" by Bowditch \cite[Prop. 3.1]{Bow} and Masur--Schleimer \cite[Theorem 3.15]{MS}.

\begin{lemma}
(Guessing geodesics lemma). Let $G$ be a graph and suppose that for every pair of vertices
$x, y \in G$ there is an associated connected subgraph $g(x, y)$ containing $x$ and $y$. If there exists an $M > 0$ such
that:
\begin{itemize}
\item[(1)] if $d_{G}(x,y) \leq 1$, then $g(x, y)$ has diameter at most $M$, and
\item[(2)] for every $x, y, z \in G$, $g(x, y)$ is contained in the $M$-neighbourhood of $g(x, z) \cup g(y, z)$,
\end{itemize}
then $G$ is $\delta$-hyperbolic, where $\delta$ depends only on $M$.
\end{lemma}

Similar to the work of Fanoni--Ghaswala--McLeay \cite{FGM}, we can show that the above lemma applies in our situation by letting $g(\alpha, \beta) = \calN_1( \calU_{G}(\alpha,\beta))$.

\begin{lemma}\label{Lem:DeltaHyp}
Assume $|\calS(\Sigma)| \geq 3$. Then $g(\alpha, \beta) = \calN_1( \calU_{G}(\alpha,\beta))$ yields connected subgraphs such that:
\begin{itemize}
\item[(1)] if $d_{\calG}(\alpha,\beta)\leq 1$, then $g(\alpha, \beta)$ has diameter at most 3, and
\item[(2)] for every $\alpha, \beta, \gamma \in \calG(\Sigma)$, $g(\alpha, \beta)$ is contained in the 3-neighbourhood of $g(\alpha, \gamma) \cup g(\beta, \gamma)$.
\end{itemize}
In particular, there is a uniform constant $\delta$ (independent of $\Sigma$ or $\calS(\Sigma)$) such that if $|\calS(\Sigma)| \geq 3$ then $\calG(\Sigma)$
is $\delta$-hyperbolic.
\end{lemma}

\begin{proof}
By Lemma \ref{Lemma:FGM6.8}, $g(\alpha, \beta) = \calN_1 (\calU_{G}(\alpha,\beta))$ is a connected subgraph of $\calG(\Sigma)$.
Suppose first that $d_{\calG}(\alpha,\beta)\leq 1$.
This implies that the grand arcs $\alpha$ and $\beta$ are disjoint.
Therefore, $g(\alpha, \beta) = \calN_1 (\calU_{G}(\alpha,\beta)) = \calN_1(\alpha) \cup \calN_1(\beta)$, which immediately implies that $g(\alpha, \beta)$ has distance at most $3$, proving $(1)$.

We now wish to show that for every $\alpha, \beta, \gamma \in \calG(\Sigma)$, the subgraph $g(\alpha, \beta)$ is contained in the 3-neighbourhood of $g(\alpha, \gamma) \cup g(\beta, \gamma)$.
This proof is the same as the proof by Fanoni-Ghaswala-McLeay, but in place of considering the precise endpoints of the grand arcs, we consider the elements of the grand splitting which the grand arc converges to.
\end{proof}

As a corollary to Lemma \ref{Lem:DeltaHyp}, we find that if every two witnesses of the grand arc graph intersect, then the grand arc graph is $\delta$-hyperbolic.

We end this section by summarizing the results of the previous 
lemmas and corollaries into a complete classification of when $\calG(\Sigma)$ is $\delta$-hyperbolic:\\

\begin{proposition} ~\label{prop:Hyperbolic}
Suppose $\calS(\Sigma)$ is finite. Then:
\begin{itemize}
\item If $|\calS(\Sigma)|\le 1$, then $\calG(\Sigma)$ is empty.
\item If $1<|\calS(\Sigma)| \le 2$, and $\Sigma$ is not the once-punctured cantor tree, then $\calG(\Sigma)$ is not $\delta$-hyperbolic.
\item If $|\calS(\Sigma)| \ge 3$ or $\Sigma$ is the once-punctured 
cantor tree, then $\calG(\Sigma)$ is $\delta$-hyperbolic.
\end{itemize}
\end{proposition}

This, together with Proposition ~\ref{prop:InfDiam}, gives a proof of Theorem ~\ref{thm:InfDiamHyperbolic}.

\section{The action of the mapping class group}\label{Sec:QuasiContinuous}

The action of the mapping class group on 
the grand arc graph is not continuous.
This can be shown by following the proof of Proposition 6.3 in \cite{FGM}. For 
example, if the maximal ends are finite-orbit ends, then the grand arc graph is precisely the omnipresent arc graph and the proof follows exactly.
However, in this section we show that the 
action of the mapping class group on a $\delta$-hyperbolic grand arc graph is quasi-continuous, a weaker notion of continuity. Using this, we show that the action on the visible boundary of a $\delta$-hyperbolic 
grand arc graph is continuous.

To begin this section, we introduce the notion of quasi-continuity.
We then prove a lemma which tells us that we can choose our witness to have some helpful properties.
This lemma will help us prove that the action of the mapping class group on a $\delta$-hyperbolic grand arc graph is quasi-continuous, and conclude that the action on the boundary is continuous.

Next, we construct examples of loxodromic actions of the mapping class group on the grand arc graph.
Finally, we will introduce the notion on stability and prove that for infinite-type surfaces which are stable, the mapping class group acts on the grand arc graph with finitely many orbits.

\subsection{Quasi-continuity \& continuous action on the boundary}
The goal of this subsection is to prove that the action of the mapping class group on a $\delta$-hyperbolic grand arc graph is quasi-continuous. Using this, we will show that the action of the mapping class group on the visible boundary of a $\delta$-hyperbolic grand arc graph is continuous.

\subsubsection*{Quasi-continuity}
We begin this subsection by defining quasi-continuity by finding an equivalent condition for continuity of a topological group on a discrete graph, and then relaxing the definition.

Throughout this section, for a group $G$ acting on a set $X$, if $U \subset G$, and $A \subset X$, then we denote $U \circ A = \{ga : g\in U, a\in A\}$.

\begin{lemma} \label{lem:continuity-conditions}
Let $G$ be a topological group acting on a discrete metric space $X$. Then the following conditions are equivalent:
\begin{itemize}
\item The action $G\times X \to X$ defined by $\pi(g,x) = gx$ is continuous.
\item For any $x\in X$ there exists an open neighbourhood $id\in U_x \subset G$ such that $U_x \circ \{x\} = \{x\}$.
\end{itemize}
\end{lemma}

\begin{proof}
Let $x\in X$ be arbitrary. If the action $\pi:G\times X \to X$ is continuous, then in particular, we must have that $\pi^{-1}(\{x\})$ is 
open. Since $(id, x) \in \pi^{-1}(\{x\})$, 
and by discreteness of $X$, there must be an open neighbourhood $U_x \times \{x\} \in \pi^{-1}(\{x\})$. By definition of $\pi$, it follows that $U_x \circ \{x\} = \{x\}$, and we're done.

Next, we suppose that for any $x \in X$, there exists an open neighbourhood $id \in U_x \subset G$ such that $U_x \circ \{ x \} = \{ x \}$. 
Fix some $x \in X$. 
We wish to show that $\pi^{-1}(\{x\})$ is open. For $(h,y) \in \pi^{-1}(\{x\}) \subset G\times X$ we let $U_y$ be an open neighbourhood of $id$ such that 
$U_y \circ \{y\} = \{y\}$. Since the group action of $G$ on itself is by homeomorphisms, it follows that 
$hU_y$ is an open neighbourhood of $h$. Moreover, $hU_y \circ \{y\} = \{x\}$, meaning that 
$hU_y \times \{y\}$ is an open neighbourhood of 
$(h,y)$ contained in $\pi^{-1}(\{x\})$. Thus, 
$\pi^{-1}(\{x\})$ is open, and since $X$ is discrete, this completes the proof.
\end{proof}

To define quasi-continuity, we first recall the definition of Hausdorff distance.
Let $Y$ and $Z$ be two non-empty subsets of a metric space $(X,d)$.
The \textit{Hausdorff distance} is defined to be
\begin{align*}
d_{X}(Y,Z) &= \max \left\{ \sup\limits_{y\in Y} d(y,Z), \sup\limits_{z \in Z} d(Y,z)  \right\}\\
\end{align*}
where $d(a,B) = \inf\limits_{b \in B} d(a,b)$ is defined for any $a\in X$ and $B\subset X$.
Using the definition of Hausdorff distance, we define the notion of \textit{quasi-continuity}:

\begin{definition}
Let $G$ be a group acting on a metric space $X$. We 
say that the action is \textit{$N$-quasi-continuous} 
if for any $x\in X$ there exists an open 
neighbourhood $id \in U_x\subset G$ such that $d_X(U_x\circ \{x\},\{x\}) \le N$.
We say that the action is \textit{quasi-continuous} if it is $N$-quasi-continuous for some $N\in \NN$.
\end{definition}

\begin{remark}
  The action of a topological group on a metric 
  space is continuous if and only if for every 
  $\eps > 0$, it is $\eps$-quasi-continuous.
\end{remark}

\begin{remark}
In the above definition, if an action is quasi-continuous with constant $N = 0$, Lemma ~\ref{lem:continuity-conditions} asserts that the action is continuous.
\end{remark}

\begin{proposition}\label{prop:QuasiContinuous}
If $\calG(\Sigma)$ is $\delta$-hyperbolic, then the action $\pi: \Map(\Sigma) \times \calG(\Sigma) \to \calG(\Sigma)$ is quasi-continuous.
\end{proposition}

Before we prove Proposition \ref{prop:QuasiContinuous}, we will prove the following technical lemma, which tells us that we are able to choose our witness to have some helpful properties:

\begin{lemma}\label{Lem:SpecificWitness}
Suppose $\calG(\Sigma)$ is nonempty. Then for any $\alpha \in \calG(\Sigma)$, there 
exists a witness $S_{\alpha} \subset \Sigma$, 
such that $S_{\alpha}\cap \alpha$ is connected.
\end{lemma}

\begin{proof}
Let $W$ be a witness, and let $e_1,e_2$ be the endpoints of $\alpha$. Let $U_1$ and $U_2$ be two flare surfaces as in Lemma 3.1 of \cite{FGM} for which $U_i\cap W =\emptyset$, $\del U_i$ is a simple closed curve, $|\del U_i \cap \alpha| = 1$, and $e_i$ lies in $U_i$. Since $\alpha$ converges to $e_1$ and $e_2$, it follows that the arc segment $\alpha' = \alpha \cap (U_1\cup U_2)^c$ is compact. Let $X$ be a regular neighbourhood of $\alpha' \cup \del U_1 \cup \del U_2$. Note that $X$ is a finite-type surface intersecting $W$, since $W$ is a witness. The surface $S_{\alpha} = W \cup X$ is then the desired surface.
\end{proof}

Using Lemma \ref{Lem:SpecificWitness}, we can now prove Proposition \ref{prop:QuasiContinuous}.

\begin{proof}[Proof of Proposition \ref{prop:QuasiContinuous}]
Let $\alpha$ be a grand arc. We will find an 
open neighbourhood, $id\in U\subset G$ such that 
$\sup_{h\in U} d(h\alpha,\alpha) < 3$. Since $\alpha$ is a grand arc, we fix a witness $W_{\alpha}$ for which $\alpha\cap W_{\alpha}$ is connected, as in Lemma ~\ref{Lem:SpecificWitness}.
Let 

\begin{align*}
U = U_{W_{\alpha}} = \{ \vphi \in \Map(\Sigma) : \vphi|_{W_{\alpha}} = id \}.
\end{align*}
To complete the proof, it suffices to show that 
$\sup_{h\in U_{W_{\alpha}}} d(h\alpha,\alpha) \le 3$. 

Let $h \in U_{W_{\alpha}}$. Note that $h$ must send any complementary component of $W_{\alpha}$ that contains a 
maximal end to itself. In particular, let $\Sigma_1$, 
$\Sigma_2$, and $\Sigma_3$ be connected components of 
$\Sigma\ssm W_{\alpha}$ each of which contain a maximal 
end, and without loss of generality, assume 
that $\alpha$ intersects only $\Sigma_1$, $W_{\alpha}$, and 
$\Sigma_2$.
Thinking of the arc $\alpha$ as a function, 
$\alpha \from (0,1) \to \Sigma$, we can orient $\alpha$ so that there exists $t_1 < t_2 \in (0,1)$, and for $t \leq t_1$ $\alpha(t) \in \Sigma_1$, and for $t \geq t_2$
$\alpha(t) \in \Sigma_2$.

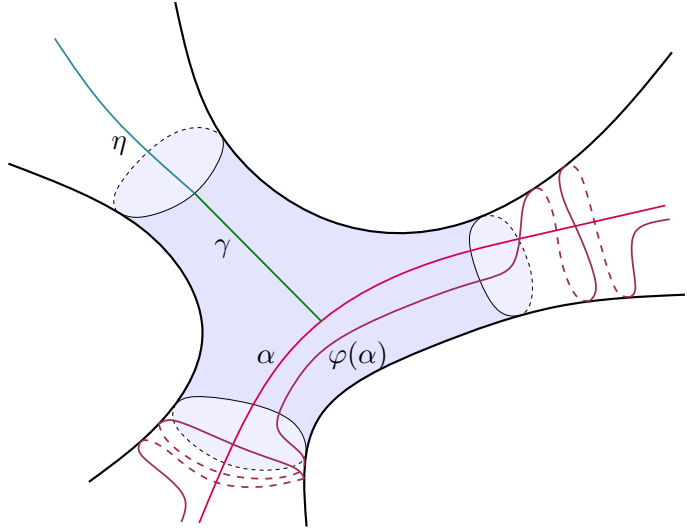
\begin{figure}[ht] 
\includegraphics[width=0.6\textwidth]{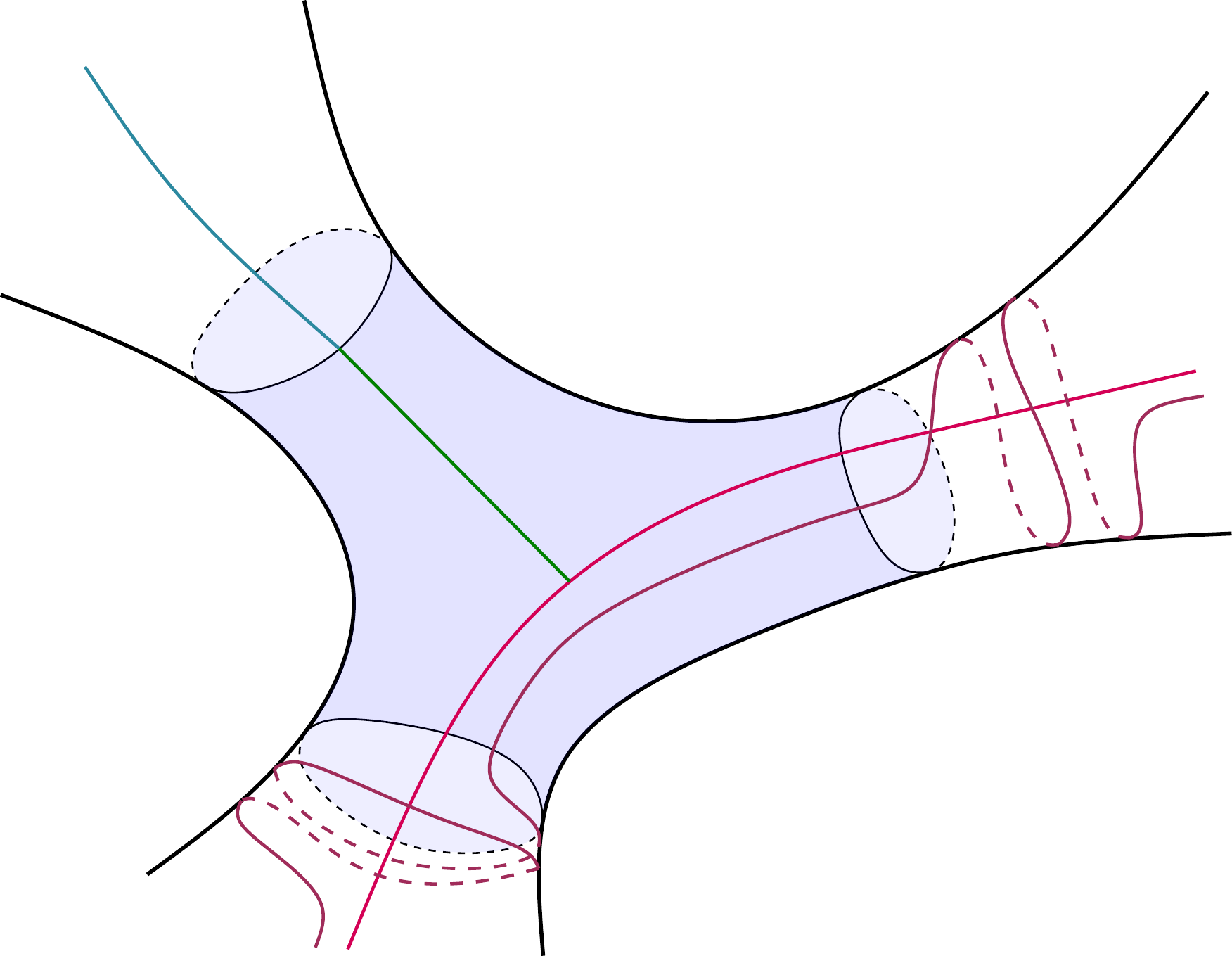}
\put(-217,143){\small{$\eta$}}
\put(-178,104){\small{$\gamma$}}
\put(-162,62){\small{$\alpha$}}
\put(-135,62){\small{$\vphi(\alpha)$}}
\caption{The components required to construct a path in the grand arc graph between a grand arc $\alpha$ and the image of $\alpha$ under a mapping class $\vphi \in U_{W_{\alpha}} \subset \Map(\Sigma)$.}
\label{fig:quasicontinuity}
\end{figure}

Let $\eta$ be a simple arc with exactly one endpoint on $\partial W_{\alpha}$, whose interior is contained in $\Sigma_3$, and which converges to a maximal 
end in $\Sigma_3$. Fix some hyperbolic metric on $\Sigma$, which induces a hyperbolic metric on $W_{\alpha}$. Let $x$ be the point on $W_{\alpha} \cap \alpha$ which is closest 
to $\eta\cap W_{\alpha}$, and let $\gamma$ be a simple arc 
whose endpoints are $\eta\cap W_{\alpha}$ and $x$. Let 
$\alpha \ssm x = \alpha_1 \sqcup \alpha_2$, 
where $\alpha_i$ intersects $\Sigma_i$. 
See Figure \ref{fig:quasicontinuity} for a picture of the above construction.
The following path proves that $\sup_{h\in U_{W_{\alpha}}} d(h\alpha,\alpha) \le 3$, which completes our proof: 
\begin{align*}
\{\alpha, \alpha_1 * \gamma * \eta, \bar\eta*\bar\gamma*\vphi(\alpha_2), \vphi(\alpha)\}
\end{align*}

\end{proof}

\subsubsection*{Continuous Action on $\del \calG(\Sigma)$}
In this subsection, we show that the action of the mapping class group on the visible boundary is continuous.
Let $X$ be a $\delta$-hyperbolic metric space, and let $\del X$ be its visible 
boundary with respect to some basepoint 
$x_0 \in X$.

\begin{theorem} 
  If $G$ acts on $X$ quasi-continuously by
isometries, then the action of $G$ on $\del X$ is continuous in the sense that the map $\pi:G\times \del X \to \del X$ is continuous.
\end{theorem}

\begin{proof}
Let $N \in \NN$ be such that $G$ acts on $X$ $N$-quasi-continuously. Let $\alpha:\NN \to X$ be a quasi-geodesic ray, and let $g_0,g_1,\hdots$ be a sequence 
  of group elements converging to $g$. Without loss of generality, assume that 
  $g_n \to \id$. We will show that the geodesics $\{g_n\circ \alpha\}_{n}$ 
  coarsely converge to $\alpha$ uniformly on compact sets. Let $K\subset \NN$ be compact, 
  and contained in $\{0,\hdots,M\}$. Let $U_{\alpha(1)} \hdots, U_{\alpha(M)}$ 
  be neighbourhoods of the identity in $G$ for which 
  $d_X(U_{\alpha(i)}\circ \alpha(i), \alpha(i)) < N$. In particular, it 
  follows that $U = \cap_{i\le M} U_{\alpha(i)}$ is an open neighbourhood of the 
  identity for which $d_X(U\circ \alpha(i),\alpha(i)) < N$ for any $i\le M$. Since 
  $\{g_n\}$ converge to the identity, it follows that after a certain point, 
  $g_n \in U$. Thus, for all $n > N_0$, we have $d_X(g_n \circ \alpha(i),\alpha(i)) \le N$ 
  for all $i \le M$, meaning that $g_n \circ \alpha$ converges to $\alpha$ on compact 
  sets up to error bounded by $N$. In other words, the sequence $\{[g_n\alpha]\}$ 
  converges to $\alpha$ in $\del X$, and the action of $G$ on $\del X$ is continuous.
\end{proof}

In particular, it follows that if $\calG(\Sigma)$ is $\delta$-hyperbolic, then the action of $\Map(\Sigma)$ on $\del \calG(\Sigma)$ is continuous.

\subsection{Loxodromic actions}
In this subsection, we construct examples of loxodromic actions of $\Map(\Sigma)$ on the grand arc graph coming from pseudo-Anosov mapping classes defined on witnesses. More precisely, we prove the following theorem:\\

\noindent \textbf{Theorem \ref{thm:lox-action}.}
\textit{Let $W$ be a witness, and let $\vphi$ 
be a pseudo-Anosov mapping class that fixes the boundary of $W$. Let $\bar\vphi\in \Map(\Sigma)$ be the homeomorphism fixing 
$W^c$ and acting as $\vphi$ on $W$. Then 
$\bar\vphi$ acts loxodromically on $\calG(\Sigma)$.}

\begin{proof}
Since $\calG(\Sigma)$ is $\delta$-hyperbolic, 
it suffices to show that the translation 
length of $\bar\vphi$ is bounded from below. To do this, we examine $\vphi$ acting on the arc complex of a witness, like in the proof of Proposition ~\ref{prop:TwoWitnessesNotHyperbolic}.

Let $\alpha_0,\alpha_1,\hdots,\alpha_n=\bar\vphi(\alpha_0)$ be a path in 
$\calG(\Sigma)$ from an arc $\alpha_0$ to its image under $\bar\vphi$. We choose $\alpha_i$ to lie in minimal position with respect to the boundary of $W$. For each $0\le i < n$, let $a_i\in \calA(W)$ be a connected component of $W\cap \alpha_i$, and let $a_n = \vphi(a_0)$. Since the $\alpha_i$'s are in minimal position, this guarantees that $a_0,\hdots,a_n$ is a path in $\calA(W)$. Since $\vphi$ acts loxodromically on $\calA(W)$, it follows that its translation length is bounded from below by some $K\in \NN$. Thus, $n\ge K$, which proves the claim.
\end{proof}

In fact, the theorem above can be used to give a much shorter proof of Proposition ~\ref{prop:TwoWitnessesNotHyperbolic}. Recall that Proposition \ref{prop:TwoWitnessesNotHyperbolic} states that if $\calG(\Sigma)$ admits two disjoint witnesses, each with Euler characteristic less than $-2$, then $\calG(\Sigma)$ is not $\delta$-hyperbolic.
The proof is as follows:

\begin{proof}[Proof of Proposition ~\ref{prop:TwoWitnessesNotHyperbolic} using Theorem ~\ref{thm:lox-action}] 
If $\calG(\Sigma)$ were 
$\delta$-hyperbolic, 
then Theorem ~\ref{thm:lox-action} would 
give two commuting loxodromic actions coming from the witnesses $W_1$ and $W_2$, contradicting hyperbolicity of the grand arc graph.
\end{proof}

While the above proof is much shorter, it also leaves out the explicit construction of a quasi-isometrically embedded copy of $\ZZ^2$ into $\calG(\Sigma)$.

\subsection{Finitely many orbits}
In this subsection, we prove that the mapping class group acts on the grand arc graph with finitely many orbits when $\Sigma$ is a surface whose maximal ends are stable in the sense of \cite{FGM, MR}:

\begin{definition}
Let $e$ be an end of a surface. A clopen neighbourhood $U$ of $e$ is \textit{stable} if for every
clopen neighbourhood $V \subset U$ there is a clopen neighbourhood $U' \subset V$ such that 
$(U, U \cap E^G(\Sigma)) \cong (U', U'\cap E^G(\Sigma))$. 
We say that an end is stable if it admits a stable neighbourhood, and is unstable otherwise.
\end{definition}

By Lemma 5.1 in \cite{FGM}, two ends are locally 
homeomorphic if and only if they are equivalent in 
the partial ordering on the space of ends defined by  Mann-Rafi. As an immediate corollary, 
it follows that there exists a homeomorphism of 
the end-space of $\Sigma$ sending a stable end $e$ to a stable end $f$ if and only if $e \sim f$. 

\begin{proposition}\label{prop:FiniteOrbits}
Let $\Sigma$ be an infinite-type surface whose maximal ends are stable.
If $|\calS(\Sigma)| \ge 2$, then
$\Map(\Sigma)$ acts on $\calG(\Sigma)$ with 
finitely many orbits.
\end{proposition}

\begin{proof}
To prove this claim, we show that if $\alpha$ and $\beta$ have endpoints which lie in the same equivalence classes of ends, then they are equal up to a homeomorphism.

Suppose first that $\alpha,\beta\in \calG(\Sigma)$ 
have exactly the same endpoints, $e$ and $f$.
Then by Lemma 3.3 of \cite{FGM} and the classification of infinite-type surfaces \cite{R}, 
it follows that $\Sigma\ssm \alpha$  
and $\Sigma \ssm \beta$ are both homeomorphic 
to the surface $\Sigma'$ which has the same genus as $\Sigma$ 
and has end space homeomorphic to $E(\Sigma)/(e\sim f)$. 
Thus, $\alpha$ and $\beta$ are the same up to homeomorphism, by \cite{AFLY}.

Next, assume that $\alpha$ and $\beta$ each have one endpoint in common, and their remaining endpoints, 
$e_{\alpha}$ and $e_{\beta}$ respectively, lie in 
the same equivalence class of maximal 
ends. Each equivalence class of 
maximal ends is either finite or a Cantor set \cite{MR}, and in particular, there must be a homeomorphism of $E(\Sigma)$ sending $e_{\alpha}$ to $e_{\beta}$. Such a homeomorphism is induced by a mapping class \cite{AFLY}.
Therefore, up to homeomorphism, we can can assume that 
the endpoints of $\alpha$ and the endpoints of $\beta$ are the same.

Finally, let $e_{\alpha}$ and $f_{\alpha}$ be the endpoints of $\alpha$ and $e_{\beta}$ and $f_{\beta}$ be the endpoints of $\beta$.
Assume that the endpoints $e_{\alpha}$ and $e_{\beta}$ lie in the same equivalence class of ends, and the endpoints $f_{\alpha}$ and $f_{\beta}$ lie in the same equivalence class of ends.
By the argument above, we may assume that, up to a homeomorphism, $\alpha$ and $\beta$ have the same endpoints.

Thus, up to applying mapping classes, there are precisely $\binom{m}{2}$ orbits of grand arcs, where $m$ is the number of equivalence classes of maximal ends, and the lemma follows.
\end{proof}

By combining Proposition~\ref{prop:FiniteOrbits} with Proposition ~\ref{prop:QuasiContinuous}, we prove Theorem \ref{thm:mcgaction} which states that the mapping class group acts by isometries on the grand arc graph, that the action is of the mapping class group on a $\delta$-hyperbolic grand arc graph is quasi-continuous, and that if $\Sigma$ is a stable surface of infinite-type then the mapping class group acts on the grand arc graph with finitely many orbits.

As a corollary of Proposition \ref{prop:FiniteOrbits}, we find that when the maximal ends are stable, the quotient space $\calG(\Sigma)/\Map(\Sigma)$ is a clique of size $\binom{m}{2}$, where $m$ is the number equivalence classes of maximal ends.
To see that the quotient is a clique, we notice that for any quadruple of elements in the grand splitting, $E_1,E_2,E_3$, and $E_4$, we can construct disjoint arcs whose four endpoints lie in $E_1,E_2,E_3$ and $E_4$.

\begin{corollary}
Suppose $\Sigma$ is an infinite-type surface whose maximal ends are stable.
If $|\calS(\Sigma)|<\infty$, then 
$\calG(\Sigma)/\Map(\Sigma)$ is a clique of 
size $\binom{m}{2}$, where $m$ is the number of equivalence classes of maximal ends.
\end{corollary}

In particular, when $\Sigma$ has exactly one equivalence class of maximal ends, then the mapping class group acts transitively on $\calG(\Sigma)$.

\bibliographystyle{plain}
\bibliography{GrandArcs}

\end{document}